\documentclass[11pt]{amsart}

  \usepackage[colorlinks, pagebackref, linkcolor=red, citecolor=blue]{hyperref}
 \usepackage{amsmath,amssymb}
 \usepackage{color}
 \usepackage{mathrsfs}
 \usepackage{graphicx}
 \usepackage{lineno}

 \def\bl#1{\textcolor{blue}{#1}}

 \numberwithin{equation}{section}
 \theoremstyle{plain}
 \newtheorem{theorem}{Theorem}[section]
 \newtheorem{lemma}[theorem]{Lemma}
  
 \newtheorem{observation}[theorem]{Observation}

 \newtheorem*{remark}{Remark}

 \newtheorem{corollary}[theorem]{Corollary}
 \newtheorem{proposition}[theorem]{Proposition}

   \theoremstyle{definition}
\newtheorem{definition} [theorem] {Definition}

  \newtheorem{example}[theorem]{Example}

 \def\bs{\boldsymbol}
 
 \def\T{ \mathbb T}
 \def\R{ \mathbb R}
 \def\H{H^\infty}

 \def\D{{ \mathbb D}}
  \def\K{{ \mathbb K}}
    
 \def\C{{ \mathbb C}}
 
 \def\N{{ \mathbb N}}

  \newcommand{\zit}[1]{(\ref{#1})}
  
 \def\bsr{\operatorname{bsr}}

 \def\e{\varepsilon}
 \def\dbar{\ov\partial}

 \def\bs{\boldsymbol}
 \def\dis{\displaystyle}
 \def\union{\cup}
 
 \def\inter{\cap}
 \def\Inter{\bigcap }
 \def\ov{\overline}

 \def\ss{\subseteq}
 \def\emp{\emptyset}

 \def\buildrel#1_#2^#3{\mathrel{\mathop{\kern 0pt#1}\limits_{#2}^{#3}}}
 
 \def\IBP{interpolating Blaschke product}

 \def\ssi{\Longleftrightarrow}
 \def\imp{\Longrightarrow}

\begin{document}

\title[Reducibility to the principal component]{%
Reducibility of invertible tuples to the principal component in 
  commutative Banach algebras%
\keywords{ 2010 Mathematics Subject Classification: Primary 46J05, Secondary 46J10 and 54C20.
Keywords: Commutative Banach algebras; exponential matrices; principal component;
 exponential reducibility;  Bass stable rank.}}


\author{%
 Raymond Mortini\and
  Rudolf Rupp
}

\address{%
  Raymond Mortini\\
  D\'{e}partement de Math\'{e}matiques et  
Institut \'Elie Cartan de Lorraine,  UMR 7502\\
  Universit\'{e} de Lorraine, Metz\\
France\\
raymond.mortini@univ-lorraine.fr
\and
  Rudolf Rupp\\
  Fakult\"at f\"ur Angewandte Mathematik, Physik  und Allgemeinwissenschaften\\
  TH-N\"urnberg\\
  Germany\\
  Rudolf.Rupp@th-nuernberg.de
} 

\maketitle

\begin{abstract}
Let $A$ be  a complex, commutative unital Banach algebra.
We introduce two notions of exponential reducibility of Banach algebra tuples and 
present an analogue to the Corach-Su\'arez result on the connection between
reducibility in $A$ and in $C(M(A))$. Our methods are of an analytical nature.
Necessary and sufficient geometric/topological conditions are given for reducibility
(respectively reducibility to the principal component of $U_n(A)$) 
whenever the spectrum of $A$ is homeomorphic to a subset of $\C^n$.
\end{abstract}


  \section{Introduction}
  
  The concepts of stable ranges  and  reducibility of invertible tuples in rings originate from  Hyman Bass's work \cite{ba} treating problems in algebraic $K$-theory.   Later on, due to work of L. Vasershtein \cite{va}, these notions also turned out to be very important in the theory of function algebras and topology
  because of their  intimate relations to extension problems.  This direction has further been developed by 
  Corach and Su\'arez, \cite{cs1}, \cite{cs}.  Function theorists have also been  interested in this subject
  and mainly computed the stable ranks for various algebras of holomorphic functions. For example,   P.W.  Jones, D. Marshall and T. Wolff \cite{jmw} determined  the stable rank of the disk algebra $A(\D)$,
 and Corach and Su\'arez \cite{cs3} the one  for the polydisk and ball algebras. The whole 
culminated in S. Treil's work on the stable rank for the algebra $\H$ of bounded
  analytic functions on the unit disk \cite{tr}. Recent work  includes investigations of stable ranks
 for  real-symmetric function algebras (see for instance \cite{mowi} and \cite{mr4}). 
The subject of the present paper is linked to the theory developped by Corach
 and Su\'arez and provides a detailed analysis of the fine structure of the set $U_n(A)$
 of invertible tuples within the realm of  commutative Banach algebras.
The main intention   is the introduction of a new
 concept, the exponential reducibility of $n$-tuples, and to present a new view
 on  the structure of the connected components of $U_n(A)$.
 
  In contrast to the work of Corach and Su\'arez (and Lin), we use an analytic framework (and methods)
instead of the powerful algebraic-topological setting. We think that this makes the
theory  accessible to a larger readership.\\ 

\subsection{Notational background and scheme of the paper}
Let  $A$ be a commutative unital Banach algebra over $\K=\R$ or $\K=\C$,
 the identity element (or multiplicatively neutral element) being denoted by $\bs 1$.
 Then  the spectrum (=set of nonzero, multiplicative $\K$-linear functionals on $A$) of $A$ is 
 denoted by  $M(A)$,
  and the set of all $n\times n$-matrices over $A$ by $\mathcal M_n(A)$. 
If $f\in C(X,\K)$, the space  of all $\K$-valued continuous functions on the topological space $X$, 
then $Z(f):=\{x\in X: f(x)=0\}$.
If $\bs f=(f_1,\dots, f_n)\in C(X,\K^n)$, then $Z({\bs f}):=\Inter_{j=1}^n Z(f_j)$ is the joint zero-set.
Moreover,  if $\bs f\in A^n$, then $|\bs f|=\sqrt{\sum_{j=1}^n |f_j|^2}$, 
$\langle\bs f,\bs g\rangle:=\bs f\cdot\bs g:=\sum_{j=1}^nf_jg_j$
and, when viewed as an element in $A^n$,   $\bs e_1:=(\bs 1,0,\dots,0)$. Finally, for $f\in C(X,\K)$,
$$||f||_\infty=||f||_X=\sup\{|f(x)|: x\in X\}.$$


Let us begin with the pertinent definitions.

  \begin{definition}

 $\bullet$ ~~ An $n$-tuple $(f_1,\dots,f_n)\in A^n$ is said to be {\it invertible} (or {\it unimodular}), 
  if there exists
 $(x_1,\dots,x_n)\in A^n$ such that the B\'ezout equation $\sum_{j=1}^n x_jf_j= {\bs 1} $
 is satisfied.
   The set of all invertible $n$-tuples is denoted by $U_n(A)$. Note that $U_1(A)=A^{-1}$.

 $\bullet$ ~~ An $(n+1)$-tuple $(f_1,\dots,f_n,g)\in U_{n+1}(A)$ is  called {\sl reducible}  (in $A$)
 if there exists 
 $(a_1,\dots,a_n)\in A^n$ such that $(f_1+a_1g,\dots, f_n+a_ng)\in U_n(A)$.
 
  $\bullet$ ~~  The {\sl Bass stable rank} of $A$, denoted by $\bsr A$,  is the smallest integer $n$ such that every element in $U_{n+1}(A)$ is reducible. 
 If no such $n$ exists, then $\bsr A=\infty$. 

\end{definition}
    

The following two results due to   Corach and Su\'arez are the key to the theory of stable ranks.

\begin{lemma}\label{clopenreduc} (\cite[p. 636]{cs1} and \cite[p. 608]{cs2}).
 Let $A$ be  a commutative, unital Banach algebra over $\K$.%
 Then, for $g\in A$, the set 
$$
\mbox{$R_n(g):=\{\bs f\in A^n: (\bs f, g)$ is reducible$\}$}
$$
 is open-closed inside the open set 
 $$I_n(g):=\{\bs f\in A^n: (\bs f, g)\in U_{n+1}(A)\}.$$
In particular,
if $\phi: [0,1]\to I_n(g)$ is a continuous map and $(\phi(0),g)$ is reducible, then $(\phi(1),g)$
is reducible.  Moreover, $R_n(g)=gA^n+U_n(A)$. 
\end{lemma}

The next assertion, which gives us a relation between reducibility in a Banach algebra $A$   and 
the associated uniform algebra $C(M(A))$ of all continuous complex-valued functions on the spectrum
$M(A)$  of $A$, actually is one of the most important  theorems in the theory of the Bass stable rank: 

\begin{theorem}[Corach-Su\'arez]\label{redu-cx=a}(\cite[p. 4]{cs}).
 Let $A$ be a commutative unital complex Banach algebra and suppose that $(f_1,\dots, f_n,g)$
 is an invertible $(n+1)$-tuple in $A$. Then $(f_1,\dots, f_n,g)$ is reducible in $A$ if and only if
 $(\widehat f_1,\dots, \widehat f_n, \widehat g)$ is reducible in $C(M(A))$.
\end{theorem}
Here is now the scheme of the paper.
In Section two we have a look at the principal components of $\mathcal M_n(A)$ and $U_n(A)$
and in Section three we give a  connection between reducibility and the
 extension of  invertible rows to invertible  matrices in the principal component of $\mathcal M_n(A)$.

In the forth section of our paper we are concerned with the analogues of the results quoted above
for our new notion of  ``reducibility of $(n+1)$-tuples in $A$ to the principal component of 
$U_n(A)$''
(see below for the definition).  
In the fifth section  we  apply these results and give  geometric/topological
conditions under which $(n+1)$-tuples in $C(X,\K)$ for $X\ss\K^n$ are reducible, respectively
reducible to the principal component of $U_n(C(X,\K))$. Let us point out that
due to Vasershtein's work, the Bass stable rank of $C(X,\K)$ is less than or equal to $n+1$;
hence every invertible $(n+2)$-tuple in $C(X,\K)$ is reducible, but in general, not every tuple
having length less than $n+1$ is reducible.  In the sixth section we apply our results to the class
of Euclidean Banach algebras. In Section 8 we give a simple proof of a result by V. Ya. Lin
telling us that a  left-invertible matrix $L$ over $A$ 
can be complemented to an invertible matrix over $A$ if and only 
if the matrix $\hat L$ of its Gelfand transforms
can be complemented in the algebra $C(M(A))$.

\section{The principal components of $\mathcal M_n(A)$ and $U_n(A)$} 

In this section we  expose for the reader's convenience several results necessary to develop
our theory, and  whose  proofs  we  could not locate in the literature (in particular for the case
of real algebras).
First, let us recall that if $A=(A,||\cdot||)$ is a commutative unital Banach algebra
over $\K$, then the principal component $Exp\, \mathcal M_n(A)$ 
(=the connected component of the identity matrix $I_n$)
of the group of invertible $n\times n$-matrices over $A$ is given by 
$$Exp\,\mathcal M_n(A)=\{ e^{M_1}\,\cdots\, e^{M_k}: M_j\in \mathcal M_n(A)\}.$$
 (see \cite[p. 201]{pal})

Our description of the connected components of the set $U_n(A)$, viewed as a topological
subspace of $A^n$, is based on the following  classical result giving a relation between 
two invertible tuples that are close to each other. An elementary proof of that result 
(excepted the addendum) is given in \cite{ru99}.

  \begin{theorem}\label{appro}
   Let $A=(A,||\cdot||)$ be a commutative unital Banach algebra over $\K$.
    Suppose that $\mathbf f=(f_1,\dots,f_n)$ is an invertible $n$-tuple
   in $A$.  Then there exists $\e>0$ such that the following is true:
   \begin{enumerate}
   \item[{}]

   For each $\mathbf g=(g_1,\dots,g_n)\in A^n$ satisfying  $\sum_{j=1}^n ||g_j-f_j||<\e$
   there is a matrix $H\in \mathcal M_n(A)$ such that
   $$ \left(\begin{matrix} g_1\\ \vdots\\ g_n\end{matrix}\right) =
 (  \exp \,H )\left(\begin{matrix} f_1\\ \vdots\\ f_n\end{matrix}\right) .$$
   \end{enumerate}
   In particular, $\mathbf g$ itself is  an invertible $n$-tuple. Moreover, if $\bs u \cdot \bs f^t=\bs 1$,
   then $$\bs u e^{-H} \cdot \bs g^t=\bs 1.$$ 
Addendum:  if $f_n=g_n$, then $H$ can be chosen so that its last row is the zero vector
   and 
   $$e^H=\left(\begin{matrix} e^M &*\\ \bs 0_{n-1}& {\bs 1} 
   \end{matrix}\right)$$
   for some matrix $M\in \mathcal M_{n-1}(A)$.
\end{theorem}

\begin{theorem}\label{compo}
Let $A$ be a commutative unital Banach algebra over $\K$.
  If $\mathbf f=(f_1,\dots,f_n)\in U_n(A)$, then the connected component, $C(\mathbf f)$, 
   of   $\;\mathbf f$ in $U_n(A)$  equals the set
   $${\mathbf f}\,\cdot\, Exp\,\mathcal M_n(A). $$
   In particular, $C(\mathbf f)$, is path-connected.
\end{theorem}

\begin{proof}
Let $\mathfrak C= {\mathbf f}\,\cdot\, Exp\,\mathcal M_n(A)$; that is 
$$ \mathfrak C  =\left\{{\mathbf f}\cdot \Big( \prod_{j=1}^k \exp\,{M_j}\Big): \;M_j\in \mathcal M_n(A), k\in \N\right\}.$$
We first note that $ \mathfrak C$ is path-connected. To see this,  let 
$$\bs g=\bs f\cdot \exp ( M_1) \cdots \exp(M_k),$$
for some $M_j\in \mathcal M_n(A)$. Then the map $\phi: [0,1]\to U_n(A)$, given by
$$\phi(t):= \bs f \cdot  \exp (t M_1) \cdots \exp(tM_k),$$
is a continuous path joining $\bs f$ to $\bs g$ within $U_n(A)$.

We claim that $\mathfrak C$ is open and closed in $U_n(A)$.  In fact, let $\mathbf g\in \mathfrak C$.
According to  Theorem \ref{appro}, there is $\e>0$  so that 
 for every $\mathbf h\in A^n$ with $$\sum_{j=1}^n||g_j-h_j||< \e,$$
  there is matrix $M\in \mathcal M_n(A)$ such that 
  $\mathbf h^t=(\exp\,M)\, \mathbf g^t$. That is $\mathbf h = \mathbf g \cdot \exp\, M^{\,t}$.
Therefore $\mathbf h \in \mathfrak C$. Hence $\bs g$ is an interior point of $\mathfrak C$.
Thus $\mathfrak C$ is open in $A^n$. Since $U_n(A)$ itself is open  in $A^n$, we conclude that
$\mathfrak C$ is open in $U_n(A)$.

To show that $\mathfrak C$ is (relatively) closed in $U_n(A)$, we take a sequence 
$(\mathbf g_k)$  in $\mathfrak C$ that converges (in the product topology of $A^n$) to $\mathbf g\in U_n(A)$.
Applying Theorem \ref{appro} again, there is $\e>0$  so that 
 for every $\mathbf h\in A^n$ with $\sum_{j=1}^n||g_j-h_j||< \e$, there is matrix $M\in \mathcal M_n(A)$, depending on $\mathbf h$,  
 such that  $\mathbf h= \mathbf g \,\cdot\,\exp\, M^{\,t}$.
   This holds  in particular for $\mathbf h =\mathbf g_k$, whenever $k$ is large.
 Thus
 $\mathbf g=  \mathbf  g_k\,\cdot\, \exp\, M_k$  for some $M_k\in \mathcal M_n(A)$, from which we conclude
 that $\mathbf g\in \mathfrak C$. Hence $\mathfrak C$ is closed in $U_n(A)$.
 
 Being open-closed and connected now implies that $\mathfrak C$ is the maximal
 connected set containing itself. Hence, with $\bs f\in \mathfrak C$, we deduce that 
 $C(\bs f)=\mathfrak C$.  
 \end{proof}

 We are now able to define the main object of this paper: 
  
 \begin{definition}
 Let $A$ be a commutative unital Banach algebra over $\K$. Then the {\it principal component
of  $U_n(A)$} is the connected component of $\bs e_1$  in $U_n(A)$ and is given by the set
 $$\mathcal P(U_n(A)):=\bs e_1 \cdot  Exp\, \mathcal M_n(A).$$
\end{definition}
 \begin{remark} 
 $\bullet$~~  If $n=1$, then $U_1(A)=A^{-1}$ and 
 $$\mathcal P(U_n(A))=\exp A:=\{e^a\colon a\in A\}.$$
 $\bullet$~~
 Let us note that in the representation $\bs e_1 \cdot Exp\, \mathcal M_n(A)$
 of the principal component of $U_n(A)$
 any other ``canonical" element $\bs e_j$, with 
 $$\bs e_j:=(0,\dots, 0,\underbrace{\bs 1}_{j}, 0,\dots, 0),$$
 is admissible, too.
 In fact,  for $i\not=j$,
$$\gamma_{i,j}(t):=(1-t) \bs e_i+t\bs e_j$$
is a path  in $U_n(A)$ joining $\bs e_i$ with $\bs e_j$, because
$$\Big< (1-t) \bs e_i+t\,\bs e_j\;,\; \bs e_i+\bs e_j\Big>={\bs 1}.$$
 Hence $\bs e_i$ and $\bs e_j$
belong to the same connected component of $U_n(A)$.
 \end{remark}

\section{Extension of invertible rows to the principal component}

 An interesting connection between reducibility and extension of rows to invertible
 matrices (resp. to matrices in the principal component)  is given in the following theorem
(see for example \cite[p. 311]{ri} and \cite[p. 1129]{quaro}). The additional property
of being extendable to finite products of exponential matrices (hence to the principal component of $\mathcal M_n(A)$), seems not to have been
considered in the literature before (as far as we know).  
\begin{theorem}\label{extendingrows}
Let $A$ be a commutative unital Banach algebra over $\K$ with unit element ${\bs 1}$.
Suppose that  $\bs u:=(f_1,\dots,f_n,g)\in U_{n+1}(A)$ is reducible.
Then  there is an invertible matrix $W\in \,\mathcal M_{n+1}(A)$ with determinant 
 ${\bs 1}$  and which is  
a finite product of exponential matrices 
such that $\bs u \;W= \bs e_1$.  In other words, $\bs u\in \mathcal P(U_n(A))$.
 Moreover, if $M=W^{-1}$, then $\bs u$ is the first row of $M$. 
 
\end{theorem}
\begin{proof}
Let $\bs f=(f_1,\dots,f_n)$. Since $\bs u=(\bs f, g)$ is reducible, 
there exists $\bs x=(x_1,\dots,x_n) \in A^n$ with $\bs  f+g\; \bs  x\in U_n(A)$.
Hence there is $\bs  y=(y_1,\dots,y_n)\in A^n$ such that \footnote{ Note that we do want the 
element ${\bs 1}-g$ here on the right-hand side.}
$$ \bs   y\cdot (\bs  f+ g\bs  x)^t={\bs 1}-g.$$
 
Consider the matrices 
$$W_1=\left[\begin{array}{ccccc}
{\bs 1}&&&&y_1\\
0&{\bs 1}&&&y_2\\
\vdots&&\ddots&&\vdots\\
0&\cdots&\cdots&{\bs 1}&y_n\\
x_1&x_2&\dots&x_n~~&\bs  y\cdot\bs  x^t +{\bs 1}
\end{array}\right]$$

$$W_2=\left[\begin{array}{cccc}
{\bs 1}&&&\\
&{\bs 1}&&\\
&&\ddots&\\
&&&\\
-(f_1+gx_1)&\cdots&-(f_n+gx_n)~~&{\bs 1}
\end{array}\right]$$

Since $$W_1=\left[\begin{array}{ccccc}
{\bs 1}&0&\cdots&\cdots&0\\
0&{\bs 1}&&&\vdots\\
\vdots&&\ddots&&\vdots\\
0&\cdots&\cdots&{\bs 1}&0\\
x_1&x_2&\dots&x_n&{\bs 1} 
\end{array}\right]~ \cdot ~
\left[\begin{array}{ccccc}
{\bs 1}&0&\cdots&&y_1\\
0&\ddots&&&y_2\\
\vdots&&\ddots&&\vdots\\
&&&{\bs 1}&y_n\\
0&\cdots&\cdots&0&{\bs 1}
\end{array}\right]=:M_1M_2,$$
it is easy to see that $W_1$ and $W_2$ are invertible matrices in $\mathcal M_{n+1}(A)$
 with determinant ${\bs 1}$
satisfying \footnote{  The matrix multiplication here is preferably done from the left to the right:
first multiply the one-row matrix with $W_1$, then go on.}
$$
(f_1,\dots,f_n,g) \;W_1 W_2=(0,\dots,0,{\bs 1}) \in A^{n+1}.
$$

Let $W_3=\Bigl[ \begin{array}{c|c|c|c|c}\bs e_{n+1}^t\; & \;(-1)^n \bs e_1^t\;&\;\bs e_2^t&\dots&
 \bs e_n^t\end{array}\Bigr] \;\;\; \in \;\;\mathcal M_{n+1}(A)$; that is (when identifying 
 $\R\cdot {\bs 1}$ with $\R$), 
 $$W_3=\left[ \begin{matrix} 0 & (-1)^n&0&0& \dots &0\\
                                                 0& 0&1&0&\dots&0\\
                                                 \vdots &\vdots & &\ddots~~&&\\
                                                  \vdots & \vdots& &&\ddots&\\
                                                 0&0 && &&1\\
                                                  1 & 0 &\dots &\dots&\dots& 0\end{matrix}\right].
$$
Note that $\bs e_{n+1}W_3=\bs e_1$ and $\det W_3=1$. If we put 
$W= W_1W_2 W_3$,  
then $W$ is invertible in $\mathcal M_{n+1}(A)$, $\det W={\bs 1}$,  and  $\bs u W = \bs e_1$, where $\bs e_1\in A^{n+1}$.
 Write $W^{-1}=\left[\begin{array}{c}
\bs w\\ V\end{array}\right]$, where $\bs  w$ is the first row.
  It is easy to see that $\bs u=\bs w$. 
Thus
$$\left[\begin{array}{c}
 \bs u\\ V\end{array}\right] W =I_{n+1}.$$
 Hence the row $\bs u $ has been extended by $V$ to an invertible matrix $M:=W^{-1}$.

Note that $M_1$ and $M_2$ in the decomposition  $W_1=M_1M_2$,
as well as  $W_2$,  have the form $I_{n+1}+N$, where $N$ is a nilpotent matrix. Hence,
 $I_{n+1}+N=e^B$ for some $B\;\in\;  \mathcal M_{n+1}(A)$ (just use an appropriate finite section of the power series expansion of the real  logarithm $\log(1+x)$).
Moreover, $W_3\;\in \;\mathcal M_{n+1}(\R)$ and $\det W_3>0$.  Thus, 
$W_3$ is a product of  exponential matrices over $\R$ \footnote{Actually,
two exponentials will suffice; see \cite{moru14}.}.  Consequently,  $W=W_1W_2W_3$
is  a finite product of exponential matrices over $A$.
     \end{proof}

\section{Reducibility to the principal component}   

\begin{definition}\label{exporedu1}
 Let $A$ be a commutative unital Banach algebra over $\K=\R$ or $\K=\C$. 
 An invertible pair $(f,g)\in U_2(A)$
 is said to be {\it reducible to the principal component  $\exp A$ of $A^{-1}$} if there exists $u,v\in A$ such that
 $$f+ug=e^v.$$ \index{principal component}
\end{definition}

\begin{center}\rule{4cm}{0.3pt}
\end{center}

It is clear that if $A^{-1}=U_1(A)$ is connected, then  the notions of ``reducibility
of pairs''  and ``reducibility of pairs to the principal component`` coincide. Our favourite example is the disk algebra $A(\D)$.
 If $U_1(A)$ is disconnected, as it is the case for the algebra $C(\T,\C)$  for example, then for every $f\in A^{-1}\setminus \exp A$, the pair $(f,0)$  is reducible, but not  reducible to the principal component of $A^{-1}$.  This notion seems  to have appeared for the first time in Laroco's work 
\cite{lar} in connection with the stable rank of $\H$.   Criteria  for various function algebras
have been established by the second author of this note in \cite{rup,ru91,ru92,ru92L}.


Now we generalize this notion to tuples, a fact that never before has been considered.  
We propose  two different settings. Here is the first one  (the second one will be dealt with
in Section \ref{sII}).

 \begin{definition}
 Let $A$ be  a commutative unital Banach algebra over $\K$. An invertible $(n+1)$-tuple
 $(\bs f,g)\;\in\; U_{n+1}(A)$ is said to be {\it reducible to the principal component of $U_n(A)$} 
 if there exists $\bs h\in A^n$ such that
 $$\bs f+g\; \bs h \in \mathcal P(U_n(A)).$$
\end{definition}

The following Proposition is pretty clear in the case of complex Banach algebras, 
since every permutation
matrix $P\in \mathcal M_n(\C)$ has a complex logarithm in $\mathcal M_n(\C)$ (see \cite{moru14}).
So what does matter here, is that we consider real algebras, too. 

\begin{proposition}\label{permutation}
Let $A$ be a commutative unital Banach algebra over $\K$ and let 
$(\bs f,g)\in U_{n+1}(A)$ be an invertible $(n+1)$-tuple in $A$ which 
 is reducible to the principal component $\mathcal P(U_n(A))$ of $U_n(A)$.
Suppose that   $\bs{\tilde f}$  is a  permutation of $\bs f$. 
Then also the tuple $(\bs{\tilde f},g)$ is reducible to the principal component of $U_n(A)$.
\end{proposition}
\begin{proof}
Without loss of generality, $n\geq 2$. 

{\bf Case 1} Let $\bs f=(f_1,\dots,f_n)$, ~~ $\bs {\tilde f}=(f_n,f_2,\dots,f_{n-1},f_1)$ and
$$S=\left(\begin{matrix} 0 &\cdots&&&{\bs 1}\\ &{\bs 1}&&&\\ &&\ddots&&\\ &&&
{\bs 1}&\\{\bs 1}&&&\cdots&0\end{matrix}\right).$$
Note that $S=S^{-1}$ and $\det S=-{\bs 1}$. The action of $S$ in
 $A\mapsto AS$ is to interchange the first and last column.
Let $W\in M_n(\R)$ be given by
$$W=\left(\begin{matrix} 0&\cdots&&&1\\
 (-1)^{n-1}~~~&0&&&\\&1&0&&&\\&&~~~\ddots&\ddots~~~&\\0&&&1&0
\end{matrix}\right).
$$
Then  $\det W=1$ and $ \bs e_n=\bs e_1 W$.  In particular $W\;\in\; Exp\; \;\mathcal M_n(\R)$.
Now, by assumption,  $\bs f+g \; \bs x=\bs e_1 M$ for some $M\;\in\; Exp \;\mathcal M_n(A)$. Hence
with $\tilde x:=x\, S$, 
\begin{eqnarray*}
\bs{\tilde f} + g\; \bs{\tilde x}& =&(\bs f+g\;\bs x) S=
 \bs e_1 \; M S\\
&=&(\bs e_n S)(MS)=\bs e_n (SMS)\\
&=& (\bs e_1 W)\; (SMS)= \bs e_1 (W \;SMS)\\
&\in&\bs e_1\cdot \mathcal M_n(A)=\mathcal P(U_n(A)),
\end{eqnarray*}
where we have used that $M\in Exp\; \mathcal M_n(A)$ if and only if 
$S^{-1}M S\;\in\; Exp\; \mathcal M_n(A)$ for every invertible matrix $S$;  just observe that
$$S^{-1}\; \big(   \prod_{j=1}^ke^{M_j}\big) S= \prod_{j=1}^k (S^{-1} e^{M_j} S)=\prod_{j=1}^k 
e^{S^{-1}M_j S}.
$$

{\bf Case 2} Let $\bs{\tilde  f}$ be an arbitrary permutation of $\bs f$.
Hence $\bs{\tilde f} = \bs f \; P$ for some permutation matrix $P$.
If $\det P>0$ then,   $P=e^{P_1}e^{P_2}$ for some
matrices $M_j\in \mathcal M_n(\R)$ (see for example \cite{moru14}).
 If $\det P<0$, then we aditionally interchange via $S$  the first
coordinate with the last one in $\bs{\tilde f}$. Let us call this new $n$-tuple $\bs F$. Then 
$\bs F= \bs f \;Q$ for some permutation matrix $Q$ with $\det Q>0$, and again
$Q=e^{Q_1}e^{Q_2}$  for some $Q_j\in M_n(\R)$. 
Now, by assumption, there exists $\bs x\in A^n$ such that
$$\bs f +g\; \bs x =\bs e_1 \; e^{M_1}\dots e^{M_k}$$
for some matrices $M_j\in \mathcal M_n(A)$. 
Hence, by multiplying at the right with $Q$,
$$\bs f \; Q + g\; \bs x Q = \bs e_1\; e^{M_1}\dots e^{M_k}e^{Q_1}e^{Q_2}.$$
Thus $(\bs F,g)$ is reducible to the principal component of $U_n(A)$.
The first case now implies that the same holds for $(\bs{\tilde f}, g).$ 
     \end{proof}

A sufficient condition for reducibility to the principal component is given in the following
technical result:

\begin{lemma}\label{tech-princi}
Let $A$ be a commutative unital Banach algebra over $\K$.  Suppose that
$(\bs f,g)\in U_{n+1}(A)$ and $n\geq 2$.  Then $(\bs f,g)$ is reducible to the principal component
of $U_n(A)$ if there exist   two vectors $\bs x\in A^n$ and  $\bs v\in U_n(A)$ such that
$\bs v$ is reducible itself with respect to some of its coordinates \footnote{
 This means that there exists $i_0$  and $a_j\in A$, $(j=1,\dots, n-1)$, such that for
  $\bs v=(v_1,\dots,v_{i_0}, \dots, v_n)$, the vector 
  $(v_1+a_1 v_{i_0}, \dots, v_{i_0-1 } + a_{i_0-1 }v_{i_0},\; v_{i_0+1 } + a_{i_0+1 }v_{i_0},\dots,
v_n+a_n v_{i_0})$ belongs to $U_{n-1}(A)$.} and 
$$ (\bs f+ \bs x \;g)\;\cdot\; \bs v^t=\bs 1.$$
\end{lemma}

\begin{proof}
Suppose that $i_0\not=n$. Then we interchange the $i_0$-th coordinate with the $n$-th coordinate
in the three vectors $\bs f, \bs x$ and $\bs v$ appearing here.  The new vectors
$\tilde{\bs f}, \tilde{\bs x}$ and $\tilde{\bs v}$ still
satisfy the B\'ezout equation
$$ (\tilde{\bs f}+ \tilde{\bs x} \;g)\;\cdot \;\tilde{\bs v}^t=\bs 1.$$
Since $(\tilde{\bs f},g)$ is reducible to the principal component of $U_n(A)$ if and only if
$(\bs f,g)$ does (Proposition \ref{permutation}),  we may assume, right at the beginning,
that $\bs v$ is reducible with respect to its last coordinate.

 By  Theorem  \ref{extendingrows},
the reducibility of the row vector $\bs v$  implies the existence of a finite product  $P_1$
of exponential matrices over $A$ such that $\bs v^t$ is the first column
of a matrix $P_1\in Exp\;\mathcal M_n(A)$. Hence (as matricial products)
$$(\bs f+\bs x\; g) \;P_1= (\bs f+\bs x\; g) \, (\bs v^t | **\;*)= (\bs 1, x_2,\dots,x_{n})$$
for some $x_j\in A$.   If we let 
$$P_2=\left(\begin{matrix} \bs 1 & -x_2&\dots&-x_{n}\\ &\bs 1& &\\ &&\ddots&\\
 &&&\bs 1 \end{matrix}\right),
$$
then $P_2\in Exp\;\mathcal M_n(A)$ (because it has the form $I_n+N$, where $N$
 is nilpotent), and
 $$(\bs f+\bs x\; g) \;P_1P_2= (\bs 1, x_2,\dots,x_{n})P_2= \bs e_1.$$
Hence  $\bs f+\bs x\; g~\in~\bs e_1\;\cdot\; Exp\;\;\mathcal M_n(A)=\mathcal P(U_n(A))$.
     \end{proof}

In order to study the reducibility to the principal component, we introduce a certain equivalence relation on the set of $n$-tuples,
reminiscent of that in \cite{cs1}. Corach and Su\'arez 
considered diagonal matrices $M$ all of
whose diagonal entries were invertible elements in $A$: 
$\bs f\buildrel\sim_{a}^{CS} \bs g \ssi\bs f-\bs g \,M\; \in\;  a \,A^n$ for such a matrix $M$.
The equivalence classes of that  relation, though, do not seem to  be compatible with the connected
components of $I_n(a)$; openness for example fails.

\begin{theorem}\label{equiclass}
 Let $A$ be  a commutative unital Banach algebra over $\K$, $a\in A$, and consider the open set
  $$I_n(a):=\{\bs f\in A^n: (\bs f, a)\in U_{n+1}(A)\}.$$
Given $\bs f, \bs g\in A^n$, define
 the relation
$$\bs f\buildrel\sim_{a}^{exp} \bs g\ssi \exists\, \bs x\in A^n, \exists B_1,\dots, B_k\in \mathcal M_n(A): 
\bs f+a\;\bs x  =\bs g\; e^{B_1}\dots e^{B_k}.$$
Then
\begin{enumerate}
\item [{\rm (1)}]~~ $\buildrel\sim_{a}^{exp}$ is an equivalence relation on $A^n$.
\item [{\rm (2)}]~~ If $\bs f\in I_n(a)$, then $[\bs f]\;\ss\; I_n(a)$, where 
$$[\bs f]:=\{\bs h\in A^n: \bs h\buildrel\sim_{a}^{exp} \bs f\}$$
~~~is the equivalence class associated with $\bs f$.
\item [{\rm (3)}]~~  If $\bs f\in I_n(a)$, then
\begin{enumerate}
\item [{\rm (3i)}]~~ $[\bs f]$ is open in $A^n$,
\item [{\rm (3ii)}]~~ $[\bs f]$ is a closed-open subset of $I_n(a)$,
\item [{\rm (3iii)}] ~~ $[\bs f]$ is a (path)-connected set within $I_n(a)$.\\
\end{enumerate}
\item [{\rm (4)}] ~~The connected components of $I_n(a)$ are the equivalence classes $[\bs f]$,
where ${\bs f\in I_n(A)}$.
\end{enumerate}

\end{theorem}

\begin{proof}
(1) $\bullet~~\buildrel\sim_{a}^{exp}$ is reflexive: just take $\bs x=\bs 0$ and $B_j=O$.

$\bullet~~\buildrel\sim_{a}^{exp}$ is symmetric: if $\bs f+a\;\bs x  =\bs g\; e^{B_1}\dots e^{B_k}$,
then
$$\bs g-a \big(  \bs x \;e^{-B_k}\dots e^{-B_1}\bigr)=\bs f\; e^{-B_k}\dots e^{-B_1}.$$

$\bullet~~\buildrel\sim_{a}^{exp}$ is transitive (here we use  that in the definition of the relation
$\buildrel\sim_{a}^{exp}$ products of exponential matrices appear; a single exponential matrix 
would not be sufficient):  let
$\bs f_1\buildrel\sim_{a}^{exp} \bs f_2$  and $\bs f_2\buildrel\sim_{a}^{exp}\bs f_3$, then there exist $\bs x_j\in A^n$ and 
$E_j\;\in\; Exp\;\;\mathcal M_n(A)$ such that
\begin{eqnarray*} \bs f_1+a\; \bs x_1 &=& \bs f_2 \,E_1=(\bs f_3 \, E_2-a\; \bs x_2) E_1.
\end{eqnarray*} 
Then
$$\bs f_1 + a (\bs x_1+\bs x_2E_1) = \bs f_3 \; E_2 E_1.$$
Hence  $\bs f_1\buildrel\sim_{a}^{exp} \bs f_3$.

(2) Let $\bs f\in I_n(a)$. Then there is $\bs x\in A^n$ such that $\bs f + a\; \bs x\in U_n(A)$.
Now if $\bs {\tilde f}\in [\bs f]$ then,   
$$\bs {\tilde f}+a\, \bs{\tilde x}= \bs {f}\, E $$
for some $E\;\in \; Exp\;\;\mathcal M_n(A)$.  Hence
$$\bs {\tilde f}+a\, (\bs{\tilde x}+ \bs x E)= (\bs f + a\;\bs x)E \in U_n(A),$$
from which we conclude that
$(\bs{\tilde  f}, a)\in U_{n+1}(A)$. In other words, $\bs{\tilde f}\in I_n(a)$. Thus $[\bs f]\ss I_n(a)$.

(3i) To show the openness of $[\bs f]$ whenever $\bs f\in I_n(a)$,    
 let $\bs h\in [\bs f]$.  By (2), $(\bs h, a)\in U_{n+1}(A)$.
We claim  that there is $\e>0$ such
that every $\bs h'\in A^n$ with $||\bs h'-\bs h||<\e$ is equivalent to $\bs f$. To see this, choose
according to Theorem \ref{appro},  $\e>0$  so small
that $(\bs h',a)=(\bs h, a) \,e^H$ for some $H\in \mathcal M_{n+1}(A)$. 
 In particular $\bs h'\in I_n(a)$.
By that same Theorem,
$H$ may be chosen so that the last column of $H$  is zero and that 
 $$e^H=\left(\begin{matrix} e^{K} &{\bs 0}^t_{n}\\ \bs x & {\bs 1}\end{matrix}\right)$$
for some $K\in \mathcal M_n(A)$ and $\bs x\in A^n$. Since
$$(\bs h',a)=(\bs h, a)\; \left(\begin{matrix} e^{K} &{\bs 0}^t_{n}\\ \bs x &
 {\bs 1}\end{matrix}\right),$$
we conclude that 
$$\bs h'= \bs h \, e^{K}+ a\,\bs x.$$
In other words, $\bs h' \in [\bs h]=[\bs f]$. Hence $[\bs f]$ is open in $A^n$. 

(3ii)  Let $(\bs h_j)$  be a sequence in $[\bs f]\ss I_n(a)$ converging to some $\bs h'\in I_n(a)$. 
As in the previous paragraph, if $n$ is sufficently large, we may conclude that 
$\bs h'\in [\bs h_j]=[\bs f]$ for $j\geq j_0$. Hence $[\bs f]$ is (relatively) closed in $I_n(a)$.
Furthermore, since $[\bs f]\ss I_n(a)$, we deduce from (3i)  that $[\bs f]$ is also open in $I_n(a)$.
 
 (3iii) Let $\bs {\tilde f}\in [\bs f]$; say   $\bs {\tilde f}+a\, \bs x =\bs f\, e^{M_1}\cdots e^{M_k}$ 
 for some $\bs x\in A^n$ and $M_j\in \mathcal M_n(A)$. Then the map $H:[0,1]\to A^n$ given by
 $$H(t)= \bs f  \;e^{tM_1}\cdots  e^{tM_k}- ta \; \bs x $$
 is a continuous path  joining $\bs f$ with $\bs {\tilde f}$.  By definition of $\buildrel\sim_{a}^{exp}$,
 each $H(t)$ is equivalent to $\bs f$; that is $H(t)\in [\bs f]$.  Thus $[\bs f]$ is path connected.\\
 
 (4) This follows immediately from (3i)-(3iii).
     \end{proof}

Here is the counterpart to Lemma \ref{clopenreduc}.
 
   \begin{theorem} \label{clopenredu-exp}
   Let $A$ be  a commutative unital Banach algebra over $\K$. Then, for $g\in A$, the set
    \begin{eqnarray*}
R^{exp}_n(g)&:=&\mbox{$\{\bs f\in A^n: (\bs f, g)$ is reducible to the principal component of $U_n(A) \}$}\\ &=& g\; A^n  +  \mathcal P(U_n(A))
\end{eqnarray*}
    is open-closed inside $I_n(g)$. In particular, if $\bs F: [0,1]\to I_n(g)$ is a continuous map
    for which  $(\bs F(0),g)$ is reducible to the principal component, 
then $(\bs F(1),g)$ is reducible to the principal component, too.
\end{theorem}

\begin{proof}
We first note that, by definition, $R^{exp}_n(g)\ss I_n(g)$ and that  the reducibility  of $(\bs f,g)\in U_{n+1}(A)$  to the principal component of $U_n(A)$ is equivalent to the assertion
that $\bs f\buildrel\sim_{g}^{exp} \bs e_1$. Thus 
\begin{equation}\label{rexp}
\bs f\in R^{exp}_n(g)\ssi \bs f\in [\bs e_1]\ssi [\bs f]=[\bs e_1].
\end{equation}
In other words, $R^{exp}_n(g)=[\bs e_1]$. The assertion then follows from Theorem \ref{equiclass}.

 Now if $\bs F:[0,1]\to I_n(g)$ is a curve in $I_n(g)$, then $C:=\bs F([0,1])$ is  connected. Since
  $\bs F(0)\in R^{exp}_n(g)$, we deduce  that
   $C\ss R^{exp}_n(g)$.
      \end{proof}

The following corollaries  are immediate (the second one is originally due to Corach and  Su\'arez \cite{cs1}).

\begin{corollary}
Let $A$ be  a commutative unital Banach algebra over $\K$ and $g\in A$. Then the following assertions are equivalent:
\begin{enumerate}
\item [(1)] ~~ Every invertible $(n+1)$-tuple $(\bs f, g)\in U_{n+1}(A)$ is reducible to the principal
component of $U_n(A)$; that is $I_n(g)= R_n^{exp}(g)$.
\item [(2)]~~  $I_n(g)$ is connected.

\end{enumerate}
\end{corollary}

\begin{proof}
Just note that by equation \zit{rexp},  $R^{exp}_n(g)=[\bs e_1]$ and that $[\bs e_1]$
is a connected set which is contained in $I_n(g)$ for every $g\in A$. The result now follows from
Theorem \ref{equiclass}.
     \end{proof}
\begin{corollary} \label{rning}
Let $A$ be  a commutative unital Banach algebra over $\K$ and $g\in A$. Then the following assertions are equivalent:
\begin{enumerate}
\item [(1)]~~  $I_n(g)=R_n(g)$;
\item [(2)]~~  Each component of $I_n(g)$ meets $U_n(A)$.
\end{enumerate}
\end{corollary}
\begin{proof}
Since the connected components of $I_n(g)$ are the equivalence classes $[\bs f]$ for 
$\buildrel\sim_{g}^{exp}$ with $\bs f\in I_n(g)$ (Theorem \ref{equiclass}),  we have the following equivalent assertions for a given $\bs f\in I_n(g)$:
\begin{enumerate}
\item [(i)]~~ $[\bs f]\inter U_n(A)\not=\emp$,
\item [(ii)]~~  there exists $\bs u\in U_n(A)$ such that $\bs u\buildrel\sim_{g}^{exp}\bs f$,
\item [(iii)]~~ there exists $\bs u\in U_n(A)$, $\bs x\in A^n$, and $B_1,\dots, B_k\in \mathcal M_n(A)$ such that
$$\bs f+g\, \bs x =\bs u \;e^{B_1}\cdots e^{B_k},$$
\item [(iv)]~~ $(\bs f,g)$ is reducible.
\end{enumerate}
     \end{proof}
Note that we actually proved a stronger  result than stated, because  the  assertions (i)-(iv) 
are valid for each individual $\bs f$.\\

The following two Lemmas are very useful to check examples upon reducibility.  They roughly say
that the reducibility of $({\bs f},g)$ depends only on the behaviour of the Gelfand transforms 
of the coordinates $f_j$ of $\bs f$ on the zero set of $\widehat g$.
We use the following notation:  $\widehat {\bs f}:=(\widehat f_1,\dots,\widehat f_n)$.

\begin{lemma}\label{reduaufE}
Let  $A$ be a commutative unital Banach algebra over $\C$. Suppose that $(\bs f,g)\in U_{n+1}(A)$.
Let $E:=Z(\widehat g)$.
If for some $\bs u\in U_n(A)$ and matrices $M_j\in \mathcal M_n(A)$
$$\sup_{x\in E}\big| \widehat {\bs f}(x) - \widehat {\bs u}(x) \;\exp \widehat M_1(x)\cdots 
\exp \widehat M_m(x)\big|<\e,$$
where $\e$ is sufficiently small,  then $(\bs f,g)$ is reducible in $A$.
\end{lemma}
\begin{proof}
Let $\delta:=\min\{ |\widehat{\bs f}(x)| \colon x\in E\}$. Since $(\bs f, g)\in U_n(A)$, we have  $\delta>0$.
Fix $\e\in \;]0,\delta/2]$, and  let $\bs b:= \bs u\; \exp M_1\cdots \exp M_m\in U_n(A)$ 
 be chosen so that
$$\sup_{x\in E}|\widehat {\bs f}(x)-  \widehat {\bs b}(x)|<\e.$$
Consider the path $\psi:[0,1]\to A^n$  given by
$$\psi(t)= (1-t)\,\bs f+ t\,\bs b.$$
On $E$ we then have the following estimates: 
\begin{eqnarray*}
\big|(1-t)\widehat{\bs f }+t\,\widehat{\bs b}\big |&=& \big|t(\widehat{\bs b}-\widehat{\bs f})+
\widehat{\bs f} \big|  \\
&\geq& |\widehat{\bs f}| -t \,|\widehat{\bs b}-\widehat{\bs f}|\\
&\geq& \delta-\delta/2=\delta/2.
\end{eqnarray*}
Hence the tuples $(\psi (t),g)$  are invertible in $A$ for every $t$. Since for $t=1$, $\psi(1)=\bs b\in U_n(A)$,  the tuple $(\psi(1),g)$ is reducible in $A$.
By Lemma \ref{clopenreduc}, $(\psi(0),g)$ then is  reducible which 
in turn implies  the reducibility of
 $(\bs f,g)$.
     \end{proof}

  \begin{lemma} \label{redu-exp-aufE}
 Let  $A$ be a commutative unital Banach algebra over $\C$. Suppose that $(\bs f,g)\in U_{n+1}(A)$.
Let $E:=Z(\widehat g)$.
If for some matrices $M_j\in \mathcal M_n(A)$
$$\sup_{x\in E}\big| \widehat {\bs f} (x) - \bs e_1\cdot \;\exp \widehat M_1(x)\cdots 
\exp \widehat M_m(x)\big|<\e,$$
where $\e$ is sufficiently small,  then $(\bs f,g)$ is reducible to the principal component 
$\mathcal P(U_n(A))$ of $U_n(A)$.
 \end{lemma}
 \begin{proof}
  Consider the path $\psi:[0,1]\to A^n$  given by
$$\psi(t)= (1-t)\,\bs f+ t\,\bs b,$$
where $ \bs b:=\bs e_1\;\cdot\;\exp M_1\cdots \exp M_m$. If 
$$0<\e<(1/2)\min\{ |\widehat{\bs f}(x)|\colon  x\in Z(\widehat g)\},$$
then $(\psi(t),g)\in U_{n+1}(A)$ for every $t\in [0,1]$.
Now  $(\psi(1),g)=(\bs b, g)$ is reducible to the principal component of $U_n(A)$ since
 $$\bs b+ 0\cdot g=\bs e_1\; \cdot \exp M_1\cdots\exp M_m\in \mathcal P(U_n(A)).$$
Hence, by   Theorem \ref{clopenredu-exp},  $(\psi(0), g)=(\bs f,g)$ is reducible to the principal component of $U_n(A)$, too.
     \end{proof}
 
We close  this   section with  our main theorem,  which is the  analogue to
 the Corach-Su\'arez result  Theorem \ref{redu-cx=a} (\cite{cs}).
It is based on the Arens-Novodvorski-Taylor theorem (\cite{ar1},\cite{no}, \cite{tay1}), a version of which
we recall here.

\begin{theorem}[Arens-Novodvorski-Taylor]\label{arnt}
 Let $A$ be a commutative unital complex Banach algebra, $X:=M(A)$ its spectrum
  and $\mathcal M_n(A)$ the Banach algebra
 of $n\times n$ matrices over $A$.
 \begin{enumerate}
 \item [(1i)]~~ Suppose that for some  $M\in \mathcal M_n(A)$ there is 
 $\underline M\in \mathcal M_n(C(X))$ such that
 ${\widehat M=\exp{\underline M}}$. \footnote{ Here $\widehat M$ is the matrix whose entries
 are the Gelfand transforms of the entries of the matrix $M$.} 
 Then $M=\exp{L_1}\cdots \exp{L_m}$ 
 for some $L_j\in \mathcal M_n(A)$.
 \item [(1ii)]~~ Let $M\in \mathcal M_n(A)^{-1}$. If $\widehat M$ belongs to the principal component of $\mathcal M_n(C(X))^{-1}$, then $M$ already belongs to the principal component  of $\mathcal M_n(A)^{-1}$.
\item  [(2)]~~ Let $\bs f\in U_n(C(X))$.  Then there exist
$\bs g\in U_n(A)$ and $\underline G_1,\dots, \underline G_m\in \mathcal M_n(C(X))$ such that  
$\bs f =\bs {\widehat g}\;\exp{\underline G_1}\cdots \exp{\underline G_m}$.
\item [(3)]~~  Let  $\bs u$ and $\bs v$ be in $U_n(A)$. Suppose that there are matrices 
$\underline G_j\in \mathcal M_n(C(X))$ such that
 $\widehat{\bs u} =\widehat{\bs v}\; \exp{\underline G_1}\cdots \exp{\underline G_m}$. Then 
$\bs u$ and $\bs v$  belong to the same connected component of $U_n(A)$.
\item [(4)]~~  The Gelfand transform induces a group isomorphism between  the quotient 
groups  
$$\mbox{$\dis \mathcal M_n(A)^{-1}/Exp\;\mathcal M_n(A)$ and 
$\mathcal M_n(C(X))^{-1}/ Exp \;\mathcal M_n(C(X))$.}$$

\end{enumerate}
\end{theorem}
Item (2), in particular,   says that every connected component of $U_n(C(X))$
contains an  element of the form $\widehat{\bs f}:=(\widehat f_1,\dots, \widehat f_n)$,
 where $\bs f\in U_n(A)$. Moreover, 
 (3) is equivalent to the assertion that if $\widehat{\bs u}$ and $\widehat{\bs v}$ can be joined
 by  a path in $U_n(C(X))$, then $\bs u$ and $\bs v$ can be joined by  a path in $U_n(A)$.
 Item (4) also   says that every element  in $\mathcal M_n(C(X))^{-1}$ is homotopic
 in  $\mathcal M_n(C(X))^{-1}$ to $\widehat M$  for some $M\in \mathcal M_n(A)^{-1}$.

\begin{theorem}\label{expreduvector}
Let $A$ be  a commutative unital Banach algebra over $\C$. Given an invertible tuple
$(\bs f, g)\in U_{n+1}(A)$, the following assertions are equivalent:
\begin{enumerate}
\item [(1)] ~~$(\bs f,g)$ is reducible to the principal component of $U_n(A)$;
\item [(2)] ~~ $\widehat{\bs f}|_{Z(\widehat g)}$ belongs to the principal component of 
$U_n(C(Z(\widehat g)))$.
\item [(3)] ~~ $(\widehat{\bs f},\widehat g)$ is reducible to the principal component of 
$U_n(C(M(A)))$.
\end{enumerate}
\end{theorem}
\begin{proof}
(1) $\imp$ (2)~~ Let $E:=Z(\widehat g)$. By assumption there is $\bs h\in U_n(A)$ such that
$$\bs u:=\bs f+g\;\bs h\in \mathcal P(U_n(A)).$$
That is, there are matrices $M_j\in \mathcal M_n(A)$ such that
$$\bs u=\bs e_1  \cdot \exp M_1\cdots \exp M_m.$$
 If we apply the Gelfand transform and restrict to $Z(\widehat  g)$, then
$$\widehat{\bs f}|_{E}=\bs e_1\cdot \big(  \exp \widehat M_1\cdots \exp \widehat M_m\big)|_{E}.$$
Hence
$\widehat{\bs f}|_{E}\in \mathcal P(U_n(C(E)))$.

(2) $\imp$ (1)~~ By assumption,  there exist  matrices $C_j\in \mathcal M_n(C(E))$ 
such that 
$$\widehat{\bs f}|_{E}=\bs e_1\cdot \exp C_1\cdots \exp C_k.$$
Let $A_E=\ov{\widehat A|_E}^{^{||\cdot||_\infty}}$ be the uniform closure of  the restriction
algebra $\widehat A|_E$ in $C(E)$. Since $Z(g)$ is $A$-convex, $M(A_E)=E$  (\cite{gam}).
Because $\widehat{\bs f}|_E\in (A_E)^n$ belongs to the principal component 
$\mathcal P(U_n(C(E)))$ of $U_n(C(E))$, which is ``generated'' by $\bs e_1$,  we conclude from 
 the Arens-Novodvorski-Taylor Theorem \ref{arnt}(3)  that $\widehat{\bs f}|_E$ belongs to the same
 component of $U_n(A_E)$ as $\bs e_1$; namely the principal component 
 $\mathcal P(U_n(A_E))$
  of $U_n(A_E)$. Hence there are matrices $B_j\in \mathcal M_n(A_E)$ such that
  $$\widehat{\bs f}|_E= \bs e_1\cdot\exp B_1\cdots \exp B_m.$$
Now, we uniformly approximate on $E$ the matrices $B_j$ by 
matrices $\widehat M_j$ with $M_j\in \mathcal M_n(A)$; say
$$\sup_{x\in E} \Big|  \widehat{\bs f}(x) - \bs e_1\;\cdot \exp \widehat M_1(x)\cdots \exp \widehat M_m(x)\Big|<\e
 $$
By Lemma \ref{redu-exp-aufE}, $(\bs f, g)$ is reducible to the principal component 
$\mathcal P(U_n(A))$ of $U_n(A)$.

(2) $\imp $ (3)  follows from  Lemma \ref{redu-exp-aufE} and (3) $\imp$ (2)
is clear.
     \end{proof}

If $n=2$, then the previous result reads as follows:
\begin{corollary}
Let $A$ be  a commutative unital Banach algebra over $\C$. Given an invertible pair
$(f, g)\in U_{2}(A)$, the following assertions are equivalent:
\begin{enumerate}
\item [(1)] ~~ There exist $a,h \in A$ such that $f+a g=e^h$;
\item [(2)] ~~ $\widehat{f}|_{Z(\widehat g)}=e^v$ for some $v\in C(Z(\widehat g))$.
\end{enumerate}
\end{corollary}


\section{Reducibility in $C(X,\K)$ with $X\ss\K^n$}

In this section we study the reducibility in $C(X,\K)$ with $X\ss\K^n$ in detail.

\begin{definition}\label{holecondi}\begin{enumerate}

\item [a)]~~ Let $K\ss \R^n$ be compact. A bounded connected component of 
$\R^n\setminus K$ is called a {\sl hole} of $K$.

\item [b)]~~ Let $K,L$ be two compact sets in $\R^n$ with $K\ss L$.  The pair $(K,L)$
is said to satisfy the {\sl hole condition} if every hole of $K$ contains a hole of $L$.
\end{enumerate}
\end{definition}

The following concepts were introduced in $\C$ by the second author in \cite{rup}.

\begin{definition}\label{bdc}
Let $K\ss \R^n$ be compact and $g\in C(K,\K)$. Then $g$ is said to satisfy the {\it boundary principle}
\index{boundary principle} if for every nonvoid open set $G$ in $\R^n$ with $G\ss K$
 the following condition holds:
$$(B_1)~~~ \mbox{If $g\equiv 0$ on $\partial G$, then $g\equiv 0 $ on $G$}.$$
\end{definition}

\begin{proposition}
Condition $(B_1)$ is equivalent to the following assertion:
$$ (B_2)~~~ \mbox{If $G$ is an open set in $\R^n$ such that $G\ss K\setminus Z(g)$, then there exists}$$
$$ \mbox{$x_0\in \partial G$ such that $g(x_0)\not=0$}.$$
\end{proposition}
\begin{proof}
Assume that $g$ satisfies $(B_1)$ and let $G\ss K\setminus Z(g)$ be open. Then $g$ cannot
vanish identically on $\partial G$ since otherwise $(B_1)$ would imply that $g\equiv 0$ on $G$.
A contradiction to the assumption that $G\inter Z(g)=\emp$. Hence $g$ satisfies $(B_2)$.

Conversely, let $g$ satisfy $(B_2)$. Suppose, to the contrary, that $g$ does not satisfy condition
$(B_1)$.  Then there is an open set $G$ in $\R^n$ with $G\ss K$,   $g\equiv 0$ on $\partial G$, but
such that $g$ does not vanish identically on $G$. Hence  
$$U:=\{x\in G: g(x)\not=0\}$$
is an open, nonvoid set in $\R^n$ which is contained in $K\setminus Z(g)$. 
But $\partial U\ss Z(g)$,
because
$$\partial U= \ov U\setminus U\ss \ov G\setminus U\ss(G\setminus U)\union \partial G\ss Z(g).$$
This contradicts condition $(B_2)$ for $U$.  Hence such a set $G$ cannot exist and we deduce
that $g$ has property $(B_1)$.
     \end{proof}

The following result gives an interesting connection between the hole condition (Definition \ref{holecondi}) and the boundary principle (Definition \ref{bdc}).

\begin{theorem}\label{hc=bp}
Let $K\ss\R^n$ be compact and $g\in C(K,\K)$. The following assertions are equivalent:
\begin{enumerate}
\item [(1)]~~ $g$ satisfies the boundary principle $(B_1)$.
\item [(2)]~~ $(Z(g), K)$ satisfies the hole condition; that is, every hole of $Z(g)$ contains
a hole of $K$ \footnote{ or which is the same, no hole of $Z(g)$ is entirely contained in $K$.}. 
\end{enumerate}
\end{theorem}
\begin{proof}
(1) $\imp$ (2)~~Suppose that there is a component $\Omega$ of $\R^n\setminus Z(g)$ with $\Omega\ss K$. 
Then $\Omega$ is open in $\R^n$ and $\partial \Omega\ss Z(g)$. 
 Condition $(B_1)$ now implies that
$g\equiv 0$ on $\Omega$ (note that by assumption $\Omega\ss K$). Hence $\Omega\ss Z(g)$. A contradiction.

(2) $\imp$ (1)~~ We show that the equivalent condition $(B_2)$ is satisfied.  
So let $G$ be open in $\R^n$ and assume
that   $G\ss K\setminus Z(g)$. Let $\Omega$ be a component of $G$. Then $\Omega$ is  bounded 
and open in $\R^n$.
 In view of achieving a contradiction, suppose   that $g\equiv 0$ on  $\partial \Omega$. Then $\partial \Omega\ss Z(g)$.
 Since $\Omega \inter Z(g)=\emp$,  we deduce  that 
 $ \Omega$ belongs to a connected component $C$ of $\R^n\setminus Z(g)$. 
 If  $\Omega$ is a proper subset of $C$, then a path connecting $z_0 \in \Omega$ and 
 $w \in C\setminus \Omega$ within $C$  would pass through  a boundary point $z_1$ of  $\Omega$. But then $g(z_1)=0$, contradicting $z_1 \in C$. Hence $\Omega=C$. 
Since $\Omega$ is bounded, we conclude that $\Omega$ is a hole of $Z(g)$.
 But $\Omega\ss G\ss K$; thus $(Z(g), K)$ cannot satisfy the hole condition (2).  This is a contradiction.
Hence there is $x_0\in \partial \Omega$  such that $g(x_0)\not=0$.  
 Since $\partial \Omega\ss \partial G$ (note that $\Omega$ is supposed to be a component of $G$),
  we are done.      \end{proof}

Let us emphasize that for  $Z(g)\ss K\ss\R^n$ the pair
 $(Z(g),K)$ automatically satisfies the hole condition (and equivalently the boundary principle 
 $(B_1)$) if $K^\circ=\emp$.
 An important class of zero-sets satisfying the equivalent conditions (1) and (2) is given
in the following example:

\begin{example}\label{coconn=b1}
 Let $K\ss\R^n$ be compact and $g\in C(K,\K)$. Then 
$Z(g)$ has property $(B_1)$ if:
\begin{enumerate}
\item [(i)]~  $\R^n\setminus Z(g)$ is connected whenever $n\geq 2$;
\item [(ii)]~  $\R\setminus Z(g)$ has exactly two  components whenever $n=1$.
\end{enumerate}
\end{example}


The next  two theorems give   nice geometric/topological  conditions for  reducibility of $n$-tuples, respectively for reducibility to the principal component  of $U_n(C(X,\K))$.
They generalize the corresponding facts for pairs developed by the second author in \cite{rup}
for certain planar compacta. 
\begin{theorem}\label{bprinc-multi}
Let $K\ss\K^n$ be compact and $g\in C(K,\K)$. The following assertions are equivalent:
\begin{enumerate}
\item [(1)]~ $C(K,\K)$ is $n$-stable at $g$, that is $(\bs f,g)$ 
is reducible for every $\bs f=(f_1,\dots, f_n)\in C(K,\K^n)$ such that 
$(\bs f,g)\in U_{n+1}(C(K,\K))$. 
\item [(2)]~  $\bs f|_{Z(g)}$ admits a zero-free extension to $K$  for all  $\bs f\in C(K,\K^n)$ with ${Z(\bs f)\inter Z(g)=\emp}$.
\item [(3)]~ $g$ satisfies the boundary principle $(B_1)$.
\item [(4)]~ $(Z(g),K)$ satisfies the hole condition.
\end{enumerate}
\end{theorem}

\begin{proof}
i) The equivalence of (1) with (2) is  well-known (see \cite{cs} or \cite{ru99}).
The equivalence of (2) with (4) is  \cite[Theorem 5.6]{mr5}, provided we identify
$(u_1+iv_1,\dots u_n+iv_n)$ in the complex-valued case with the real-valued
 $(2n)$-tuple $(u_1,v_1,\dots, u_n,v_n)$.
The equivalence of (3) with (4) is  Theorem \ref{hc=bp}.
     \end{proof}

 \begin{theorem}\label{reduprinci-vec}
 Let $K\ss\K^n$ be compact and $g\in C(K,\K)$. The following three assertions are equivalent:
 \begin{enumerate}
 \item [(1)]~  $(\bs f,g)$ is reducible to the principal component of $U_n(C(K,\K))$ 
 for every $\bs f\in C(K,\K^n)$ such that $(\bs f,g)\in U_{n+1}(C(K,\K))$;
 \item [(2)]~ there exist matrices $B_j\in \mathcal M_n(C(Z(g),\K))$ such that 
 $$\bs f|_{Z(g)} =\bs e_1\;\cdot \; e^{B_1}\cdots e^{B_k}$$
  for every $\bs f\in C(K,\K^n)$ with $Z(\bs f)\inter Z(g)\not=\emp$;
 \item [(3)]~  $Z(g)$ has no holes in $\K^n$.
\end{enumerate}
\end{theorem}
 \begin{proof}
First  we note that (1) and (2) are equivalent in view of Theorem \ref{expreduvector}. Since we have to deal here only with the special case $C(K,\K)$, the following  simple proof is available:\\

 (1) $\imp$ (2) ~~ If $\bs f + g\; \bs h\; \in\; \mathcal P(U_n(C(K,\K)))$ for some $\bs h\in C(K,\K^n)$
  then, by using the representation of the principal component,
  $$\bs f + g\,\bs h= \bs e_1\; \cdot\; e^{B_1}\cdots e^{B_k}$$
  for some matrices $B_j\;\in\; \mathcal M_n(C(K,\K))$.   Restricting this identity
   to $Z(g)$ yields the assertion (2).
  
  (2) $\imp$ (1) ~~ This follows from  Lemma \ref{redu-exp-aufE}.

  (2) $\imp$ (3) ~~    Suppose, to the contrary, that $Z(g)$ admits a bounded complementary component ${G\ss \K^n}$. Then $\partial G\ss Z(g)$. Let $\bs a\in G$ and $\bs f (\bs z)= \bs z-\bs a$,
  $\bs z\in Z(g)$. 
  Since $(\bs f, g)\in U_{n+1}(C(Z(g),\K))$, there exist by hypothesis (2) 
  matrices $B_j\in \mathcal M_n(C(Z(g),\K))$ such that 
  $$\bs f|_{Z(g)}=\bs e_1\;\cdot \; e^{B_1}\cdots e^{B_k}.$$
  Extending via Tietze's result  the matrices $B_j$ continuously to $\K^n$,
  would yield a zero-free extension of the  $n$-tuple $(\bs z-\bs a)|_{Z(g)}$ to $\ov G$.
  This contradicts  a corollary to Brouwer's fixed point theorem (see \cite[Chapter 4]{bu}).
  
   (3) $\imp$ (2) ~~    By a standard result in vector analysis (see for example 
   \cite[Corollary 5.8] {mr5}), the connectedness of $\K^n\setminus Z(g)$ implies
   that the invertible tuple $\bs f|_{Z(g)}\in U_{n}(C(Z(g),\K))$ admits a 
   zero-free extension  $\bs F$ to $\K^n$.
Let ${\bf B}\ss\K^n$ be  a closed ball whose interior contains $X$. 
Note that $\bf B$ is a contractible Hausdorff space.
 Hence, the set $U_n(C(\bf B,\K))$ is connected.
Thus, $\bs F|_{\bf B}=\bs e_1\; \cdot \;\;e^{B_1}\cdots e^{B_k}$ for some 
matrices $B_j\in \mathcal M_n(C(\bf B,\K))$.
Restricting to  $Z(g)$ yields the assertion (2).
     \end{proof}

\section{ Reducibility in Euclidean  Banach algebras }

Let us call a complex commutative  unital Banach algebra $A$ a {\it Euclidean  Banach algebra}
if the spectrum $M(A)$ of $A$ is homeomorphic to a compact set in $\C^n$. This class
of algebras includes every finitely generated Banach algebra   over $\C$, for example
the algebras  
$$P(K)=\ov{\C[z_1,\dots,z_n]\;\big|_K}^{\;||\cdot||_K}$$ 
and certain algebras  \footnote{ We do not know whether every algebra  $A(K)$ is finitely generated.} of type
$$\mbox{$ A(K)=\{f\in C(K,\C): f$  holomorphic in $K^\circ\}$},$$
$K\ss\C^n$ compact (for example $K=\ov\D^n$ or $K=\bs B_n$, the closed unit ball in $\C^n$).
Using the Arens-Novodvorski-Taylor theorem we can now  generalize Theorems 
\ref{bprinc-multi} and \ref{reduprinci-vec} to Euclidean Banach algebras.

\begin{theorem}
Let $A$ be  a  Euclidean Banach algebra  with spectrum $X\ss\C^n$ and let $g\in A$.  Then the assertions (1)-(3), respectively (4)-(5),  are equivalent:
\begin{enumerate}
\item [(1)]~ $(\bs f, g)$ is reducible for every 
$n$-tuple $\bs f\in A^n$ with $(\bs f,g)\in U_{n+1}(A)$.
\item [(2)]~ $(Z(\widehat g),X)$ satisfies the hole condition.
\item [(3)]~ $Z(\widehat g)$ satisfies the boundary principle in $\C^n$.\\

\item [(4)]~ $(\bs f, g)$ is reducible to the principal component of $U_{n}(A)$ for every 
$n$-tuple $\bs f\in A^n$ with $(\bs f,g)\in U_{n+1}(A)$.
\item [(5)]~ $Z(\widehat g)$ has no holes in $\C^n$.
\end{enumerate}
\end{theorem}

\begin{proof}

 First we note that by Theorem \ref{hc=bp}, (2) and (3) are equivalent. 
 
(1)   $\imp$ (2)~~ \bl{(resp. (4)   $\imp$ (5)~~)}
  By Theorem \ref{bprinc-multi} \bl{(resp. Theorem \ref{reduprinci-vec})}
 we need to show that for every 
$\bs f\in C(X,\C^n)$ with $Z({\bs f}) \inter Z(\widehat g)=\emp$, the $(n+1)$-tuple 
$(\bs f,\widehat g)$  is reducible in $C(X,\C)$ \bl{(resp. reducible to 
$\mathcal P(U_n(C(X,\C)))$\;)}. Let $E:=Z(\widehat g)$. Consider the algebra
$B:=\ov{\widehat A|_E}^{\;||\cdot||_E}$, that is the uniform closure of the restriction algebra 
$\widehat A|_E$ to $E$. Since $E$ is the zero set of the Gelfand transform of a function in $A$, 
$E$ is $A$-convex and so the spectrum, $M(B)$, of $B$ coincides with $E$ (see \cite{gam}). 
Note that $\bs f|_E\in U_n(C(E,\C))$.
By the Arens-Novodvorski-Taylor Theorem \ref{arnt} (2), there exist
$\bs h\in U_n(B)$ and $\underline G_1,\dots, \underline G_m\in \mathcal M_n(C(E,\C))$ such that  
$$\bs f|_E =\bs h\;\exp{\underline G_1}\cdots \exp{\underline G_m}.$$
In particular $|\bs h|\geq \delta>0$ on $E$. By Tietze's extension theorem, we may assume
that $\underline G_j\in \mathcal M_n(C(X,\C))$.
Using the  definition of $B$,  choose  $\bs a\in A^n$ so that   \footnote{Here
$||M||_{HS}$ is the Hilbert-Schmidt norm.}

\begin{equation}\label{klein}
\sup_{x\in E}| \bs h(x)- \bs{\widehat a}(x)|<\frac{\e} {||M||_{HS}},
\end{equation}
 where  $M:= \exp{\underline G_1}\cdots \exp{\underline G_m}$ and where
$\e$ is so small that $Z(\bs a)\inter Z(\widehat g)= Z(\bs a)\inter E=\emp$. Thus 
$(\bs a,g)\in U_{n+1}(A)$.  The hypothesis (1) \bl{(resp. (4)}
implies that $(\bs a,g)$ is reducible in $A$ \bl{(resp. reducible to the principal component of $U_n(A)$)}.
That is, there is  $\bs x\in A^n$ such that  $\bs u:=\bs a+ \bs x \, g\in U_n(A)$ \bl{(resp. 
$\bs u\;\in \;\mathcal P(U_n(A))$)}.  In particular,
$\widehat{\bs u}=\widehat{\bs a}$ on $E$.  Since  $|\bs v M|\leq|\bs v|\cdot ||M||_{HS}$
 for every vector $\bs v$, we have the following estimates on $E$:

\begin{eqnarray*}
\big| \bs f|_E-  \widehat{\bs u}\;\exp{\underline G_1}\cdots \exp{\underline G_m}\big| &=&
\big| \bs f|_E-  \widehat{\bs a}\;\exp{\underline G_1}\cdots \exp{\underline G_m} \big| \\
&=&
\big|(\bs h-\widehat{\bs a})\;\exp{\underline G_1}\cdots \exp{\underline G_m} \big| \\
&\leq& | \bs h-\widehat{\bs a}|\, ||M||_{HS}<\e.
\end{eqnarray*}
\medskip

Since $\widehat{\bs u}\;\in\; U_n(C(X,\C))$ \bl{(resp. 
$\widehat{\bs u}\;\in\; \mathcal P(U_n(C(X,\C)))$)}, we deduce from  Lemma \ref{reduaufE} \bl{(resp.  Lemma \ref{redu-exp-aufE})}, applied to the algebra 
$C(X,\C)$, that
$(\bs f, \hat g)$ is reducible in $C(X,\C)$ \bl{(resp. reducible to the principal component
of $C(X,\C)$} (whenever $\e>0$ is small).  

(2) $\imp$ (1)~~ \bl{(resp. (5) $\imp$ (4)~~)}
Let $(\bs f,g)\in U_{n+1}(A)$. Hence  $\widehat{\bs f}$ and $\widehat g$
 have no common zeros on $X$.  By Theorem \ref{bprinc-multi}, \bl{(resp. Theorem
  \ref{reduprinci-vec})}
  hypothesis (2) \bl{(resp. (5))} and the assumption $M(A)=X$ imply
 that $(\widehat{\bs f}, \widehat  g)$ is reducible in $C(X,\C)$ \bl{(resp. reducible to the principal
 component  of $U_n(C(X,\C))$)}.
   By the Corach-Su\'arez Theorem \ref{redu-cx=a} 
   \bl{(resp.  Theorem \ref{expreduvector})}
    $(\bs f, g)$ is reducible in $A$ \bl{(resp. reducible to the principal component of $U_n(A)$)}.
     \end{proof}

\section{Exponential reducibility II}\label{sII}

Here we introduce our second notion of exponential reducibility.

 \begin{definition}
Let $A$ be a commutative unital Banach algebra over $\K$ with identity element ${\bs 1}$.
Given $$(\bs a,g):=(a_1,\dots,a_n,g)\in U_{n+1}(A),$$
we call $(\bs a,  g)$  {\it exponentially reducible} \index{exponential reducibility} if
there exists $x_j, b_j\in A$ such that 
$$\sum_{j=1}^n  e^{x_j}(a_j+b_j g)={\bs 1}.$$
\end{definition}
\begin{observation}
~ Let $A$ be  a commutative unital Banach algebra  over $\K$  such that
\begin{enumerate}
\item [(1)]~ $\bsr A=1$,
\item[(2)]~ $U_1(A)$ is connected.
\end{enumerate}
Then every invertible pair $(a,g)\in A$ is exponentially reducible.
\end{observation}
\begin{proof}
By (1), $a+ bg\;\in\; U_1(A)$ for some $b\in A$.  Since
$U_1(A)=\exp A$, we arrive at $a+bg=e^x$ for some $x\in A$. This is of course equivalent to say
that $e^{-x}(a+bg)={\bs 1}$.  
     \end{proof}

The following result gives a relation between exponential reducibility and reducibility
to the principal component.

\begin{proposition}
Let  $A$ be  a commutative unital Banach algebra  over $\K$ and let $(\bs a,g)\in U_{n+1}(A)$.
Suppose that $(\bs a,g)$ is exponentially reducible. Then $(\bs a,g)$ is reducible to the principal component of $U_n(A)$. 
\end{proposition}
\begin{proof}
By assumption,  $\sum_{j=1}^n  e^{x_j}(a_j+b_j g)={\bs 1} $ for some 
$\bs v:=(e^{x_1},\dots, e^{x_n})\; \in\; U_n(A)$ and $\bs b:=(b_1,\dots,b_n)\in A^n$;
that is
$$(\bs a+ \bs b\, g) \cdot \bs v^t={\bs 1}.$$ 
Since $\bs v\;\in\; U_n(A)$ is reducible in $A$ (if $n\geq 2$),
we deduce from  Lemma \ref{tech-princi} that 
$$\bs a + \bs b\; g \;\in\; \bs e_1\;\cdot\; Exp\;\mathcal M_n(A)=\mathcal P(U_n(A)).$$
In other words, $({\bs a},g)$ is reducible to the principal component of $U_n(A)$.
The case $n=1$ is obvious.
     \end{proof}

{\bf Remark}
Whereas   for invertible pairs both notions coincide, 
 exponential reducibility of tuples of length at least three is, in general,  a much stronger 
 requirement than reducibility to the principal component. 
 As an example, we take the disk algebra $A(\D)$. 
 Let $(z,f)\in U_2(\D)$ be  an invertible pair that is not totally reducible 
 (this means that there do not exist two invertible functions $u$ and $v$ in $A(\D)$ such that
 $uz+vf=1$, see \cite{moru92} for the existence). Then 
 $(z, f, 0)\in U_3(A(\D))$. Since $\bsr A(\D)=1$, $U_2(A(\D))$ is connected 
(\cite{cs}),  and so $U_2(A(\D))$ coincides with its principal component.
 In particular  every invertible triple in $A(\D)$ is reducible to the principal component of $U_2(A(\D))$.  On the other hand, $(z,f,0)$ cannot be exponentially reducible, since
 otherwise
 $$e^{a_1(z)} (z+ b_1(z) \cdot 0) + e^{a_2(z)}(f(z) + b_2(z) \cdot 0)=1$$
 for some functions $a_j,b_j\in A(\D)$.
 But an equation of the form $e^{a_1(z)} z+ e^{a_2(z)} f(z)=1$
 is not possible because by our choice, $(z,f)$ is not totally reducible.

Examples of exponentially reducible tuples appeared in \cite{lar} (for pairs)  and  \cite{mo92}
(for tuples):
\begin{example}
Let $f_j\in\H$ and let $b$ be an \IBP. Suppose that $(f_1,\dots,f_n,b)\in U_{n+1}(\H)$.
Then $(f_1,\dots,f_n,b)$ is exponentially reducible.
\end{example}

We do not know yet a characterization  of the exponentially reducible tuples
(in none of the standard algebras).

\section{Complementing left-invertible matrices}

Based on the Arens-Novodvorski-Taylor theorem,  we conclude our paper by giving
  a simple  analytic  proof of  a result by V. Ya. Lin \cite[p. 127]{lin} concerning extension
   of left invertible matrices. Although this proof seems to be known among the specialists
   in the field, it never appeared explicitely in print (see also  the footnote in \cite[p. 127]{lin}).
One may view this result as a companion result to Theorem \ref{expreduvector} and Theorem  \ref{extendingrows}.

\begin{theorem} [Lin]
Let $A$ be a commutative unital Banach algebra over $\C$. Then a left-invertible matrix $L$
over $A$ can be complemented/extended to an invertible matrix over $A$ if and only if $\hat L$
\footnote{ Here $\hat L=(\widehat {a_{i,j}})$ is the matrix formed with the Gelfand-transforms
of the entries of $L$.}
can be complemented in the algebra $C(M(A))$.
\end{theorem}
\begin{proof} 
It is sufficient to consider invertible rows (see \cite[p. 345/346]{vid}). So 
let $\bs a\in U_n(A)$. Suppose that there exists an invertible matrix 
$M\in \mathcal M_n(C(M(A)))$ such that
$\bs{\widehat a}= \bs e_1 M$; that is, $\widehat{\bs a}$ is the first row of $M$.

 By the Arens-Novodvorski-Taylor Theorem \ref{arnt} (4),
there is $Q\in \mathcal M_n(A)^{-1}$ such that $\widehat Q$ is homotopic in $M_n(C(M(A)))^{-1}$ to 
$M$. Hence, for some $G_j\in M_n(C(M(A)))$, 
$$M= \widehat Q \;e^{G_1}\dots e^{G_k},$$ 
and so
$$\bs{\widehat a}= (\bs e_1  \widehat Q)\; e^{G_1}\dots e^{G_k}.$$
Let $\bs b$ be the first row of $Q$; that is $\bs b=\bs e_1 Q$.  Then $\widehat{\bs  b}=\bs e_1 \widehat Q$.

Since $\bs b\;\in\; U_n(A)$, we see that $\bs{\widehat a}$
and $\bs{\widehat b} $ belong to the same component of $U_n(C(M(A)))$. Hence, 
by another application of the Arens-Novodvorski Theorem  \ref{arnt} (3),  $\bs a$ and $\bs b$
belong to the same component of $U_n(A)$.   Thus, there exist $H_j\in \mathcal M_n(A)$ such that
$$
\bs a= \bs b\; e^{ H_1}\cdots e^{ H_k}.
$$
Consequently,
\begin{eqnarray*}
\bs a&= &(\bs e_1  Q)\;  e^{ H_1}\cdots e^{H_k}\\
&=& \bs e_1 \; R
\end{eqnarray*}
for some $R\in \mathcal M_n(A)^{-1}$.
     \end{proof}

{\bf Acknowledgements}
 We thank Alexander Brudnyi for valuable comments on the Arens-Novodvorski-Taylor Theorem \ref{arnt} and the referee for his suggestions.

\end{document}